\documentclass[reqno,oneside,a4paper,11pt]{amsart}
\usepackage[body={145mm,225mm}]{geometry}
\usepackage{amsthm,framed,amssymb,amsmath,verbatim,wrapfig}
\usepackage[english]{babel}
\usepackage{color}
 \usepackage[pdftex]{graphicx}
\usepackage[natbibapa,nodoi]{apacite}  % erforderlich bei Verwendung von apacite.sty

\newtheorem{trm}{Theorem}[section]

\newtheorem{lem}[trm]{Lemma}
\newtheorem{cor}[trm]{Corollary}
\theoremstyle{definition}
\newtheorem{de}[trm]{Definition}
\newtheorem{rem}[trm]{Remark}

\let\oldmarginpar\marginpar\renewcommand\marginpar[1]{\-\oldmarginpar[\raggedleft\footnotesize #1]{\raggedright\footnotesize #1}}	%Randnoten kleiner

\newcommand{\la}{\langle}
\newcommand{\ra}{\rangle}
\newcommand{\rah}{\rangle_\H}
\newcommand{\ran}{\operatorname{ran}}

\renewcommand{\d}{\,\mathrm{d}}
\newcommand{\pd}{\partial}

\renewcommand{\i}{{\mathrm{i}}}
\renewcommand{\phi}{\varphi}
\renewcommand{\epsilon}{\varepsilon}
\renewcommand{\theta}{\vartheta}
\newcommand{\C}{\mathbb{C}}
\newcommand{\R}{\mathbb{R}}

\newcommand{\N}{\mathbb{N}}
\renewcommand{\H}{\mathcal{H}}

\renewcommand{\Im}{{\mathrm{Im}}\,}
\renewcommand{\Re}{{\mathrm{Re}}\,}

\numberwithin{equation}{section}

\begin{document}
\title{Closed-loop Stability Analysis of a Gantry Crane with Heavy Chain and Payload}
\begin{abstract}
 In this paper, we analyze a systematically designed and easily tunable backstepping-based boundary control concept developed by \cite{TWK06} for a gantry crane with heavy chain and payload. The corresponding closed-loop system is formulated as an abstract evolution equation in an appropriate Hilbert space. Non-restrictive conditions for the controller coefficients are derived, under which the solutions are described by a $C_0$-semigroup of contractions, and are asymptotically stable. Moreover, by applying Huang's theorem we can finally even show that under these conditions the controller renders the closed-loop system exponentially stable.
\end{abstract}
% \author{Dominik St\"urzer \and Anton Arnold \and Andreas Kugi}
%
\date{\today}
\author[D. St\"urzer]{Dominik St\"urzer} \address{Institute for Analysis and
 Scientific Computing, Technische Universit\"at Wien, Wiedner
 Hauptstra\ss{}e 8-10, 1040 Vienna, Austria}
\email{dominik.stuerzer@gmail.com}

\author[A. Arnold]{Anton Arnold} \address{Institute for Analysis and
 Scientific Computing, Technische Universit\"at Wien, Wiedner
 Hauptstra\ss{}e 8-10, 1040 Vienna, Austria}
\email{anton.arnold@tuwien.ac.at}

\author[A. Kugi]{Andreas Kugi} \address{Automation and Control Institute, Technische Universit\"at Wien, Gusshausstra\ss{}e 27-29, 1040 Vienna, Austria}
\email{kugi@acin.tuwien.ac.at}

\maketitle
% \centerline{\bf D.~St\"urzer\footnote{Institut f\"ur Analysis und Scientific Computing, Technische Universit\"at Wien,
% Wiedner Hauptstr.\ 8, A-1040 Wien, Austria, e-mail: dominik.stuerzer@tuwien.ac.at},
% A.~Arnold\footnote{Institut f\"ur Analysis und Scientific Computing,
% Technische Universit\"at Wien,
% Wiedner Hauptstr.\ 8,
% A-1040 Wien, Austria,
% e-mail: anton.arnold@tuwien.ac.at}, A.~Kugi\footnote{Institut f\"ur Automatisierungs- und Regelungstechnik,
% Technische Universit\"at Wien,
% Gu\ss{}hausstr. 25-29,
% A-1040 Wien, Austria,
% e-mail: andreas.kugi@tuwien.ac.at}}
% %\maketitle
% \vspace{0.9em}
% {\small\centerline{\today}}

%\tableofcontents
\section{Introduction}\label{sec:1}

This paper deals with the rigorous stability analysis of a control concept presented by \cite{TWK06} applied to the infinite-dimensional model of a gantry crane with heavy chain and payload. The model consists of a cart of mass $m_c$, which moves horizontally along a rail, a heavy chain of length $L$ with mass per length $\rho$, attached to the cart\footnote{Note that here only a single chain is considered, unlike the pair of parallel chains as used by \cite{TWK06}. This change corresponds to the substitution $\rho\to\rho/2$.}, and a payload of point mass $m_p$ at its end. The chain is assumed to be inextensible and perfectly flexible. For the derivation of the equations of motion, it is further assumed that no friction occurs in the system. The force $F$ acting on the cart serves as the control input. The situation is sketched in Figure \ref{fig1} on the left-hand side.

\begin{figure}[ht]
\includegraphics[width=0.7\textwidth]{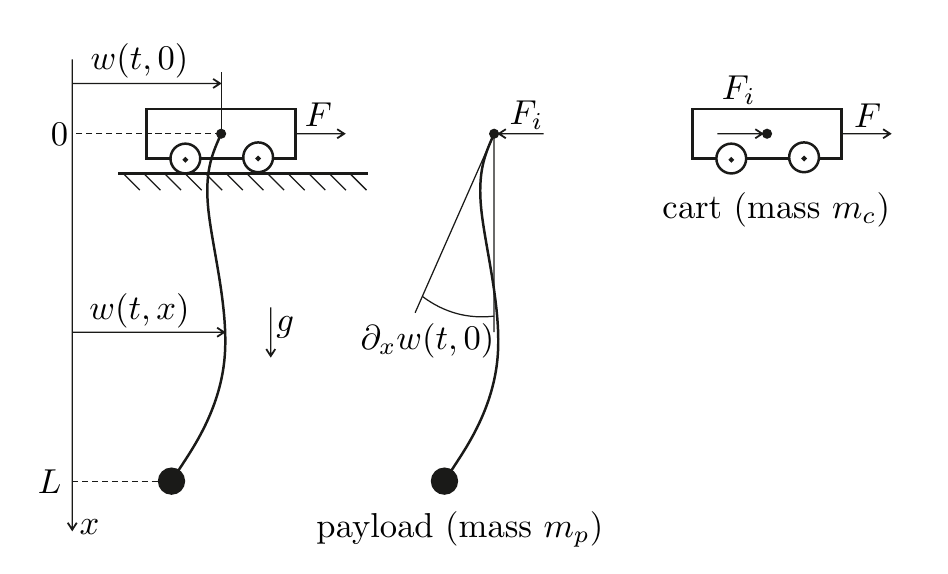}
\caption{Gantry crane with heavy chain and payload: Schematics (left) and representation of the internal force $F_i$ (right).}\label{fig1}
 \end{figure}

Let $w(t,x)$ denote the horizontal chain position. Then, under the assumption that the chain slopes $\pd_x w(t,x)$ remain sufficiently small for all $t>0$, the dynamics of the system are described by the following wave equation with dissipative, higher order boundary conditions %
\citep[see, e.g.,][]{TWK06,Petit:Rouchon}, \citep{Mifdal:1997}%
%see, e.g., \cite{TWK06}, \cite{Petit:Rouchon}, {\color{blue} \cite{Mifdal:1997}}%
\begin{subequations}\label{hamilton}
\begin{align}
	\rho\pd_t^2w(t,x)-\pd_x(P(x)\pd_x w(t,x))&=0,\label{sub:1}\\
	m_p\pd_t^2w(t,L)+P(L)\pd_xw(t,L)&=0,\label{sub:3}\\
	m_c\pd_t^2w(t,0)-P(0)\pd_xw(t,0)&=F(t).\label{sub:2}
\end{align}
\end{subequations}
The function $P(x)$ represents the tension in the chain at height $x$, given by $P(x)=g[\rho(L-x)+m_p]$, where $g$ denotes the gravitational acceleration. Note that $P\ge g\,m_p>0$ holds uniformly on $[0,L]$. In the following, it is only required that $P\in H^2(0,L)$ and that $P(x)\ge P^0>0$ holds uniformly on $[0,L]$ for some constant $P^0$. Thus, the density of the chain does not need to be constant, as it was the case for instance in \citep{TWK06}. Moreover, the following notation $v:=\pd_t w$ will be used in the sequel.

\medskip

For the system \eqref{hamilton}, many different control laws can be found in the literature %
\citep[see, e.g.,][]{Conrad:Mifdal,Mifdal:1997,Coron:Novel,TWK06},
{\color{black} with $v=w_t$.}
%\cite{Conrad:Mifdal}, \cite{Mifdal:1997}, \cite{Coron:Novel}, \cite{TWK06}.
Basically, the common structure of these controllers looks like
\begin{equation}\label{control_law_general}
F\left(  t\right)  =\vartheta_{1}v\left(  t,0\right)  +\vartheta_{2}\partial_{x}v\left(  t,0\right)  +\vartheta_{3}w\left(  t,0\right)+\vartheta_{4}\partial_{x}w\left(  t,0\right)  \text{, }
\end{equation}
but they differ in the fact which parameters $\vartheta_{j}$, $j=1,\ldots,4$ are equal to zero, which conditions have to be fulfilled to render the closed-loop system asymptotically (or even exponentially) stable, and how the controller parameters can be systematically tuned.

Thus, for instance \citep{Conrad:Mifdal} show that a (simple) passive controller with $\vartheta_{2}=\vartheta_{4}=0$ in (1.2) already ensures asymptotic stability of the closed-loop system. While this is a nice result from a theoretical point of view, this controller is not able to damp vibrations of the chain in case of stick-slip effects in the cart, which are always present in the real experiment. To make this clear, let us assume that the cart is in a sticking position, which also entails that $v\left(  t,0\right)  =0$ and $w\left(  t,0\right) = c$, with a constant $c$, then the cart will not move as long as the absolute value of the sum of the internal force in the pivot bearing carrying the chains $F_{i}=P\left(  0\right)  \partial_{x}w\left(t,0\right)$ and the input force $F(t)$, see \eqref{sub:2}, is smaller than the sticking friction. As a consequence the chain keeps on vibrating, but the cart stands still and the control law \eqref{control_law_general} with $\vartheta_{2}=\vartheta_{4}=0$ only
produces a constant input force $F\left(t\right)=\vartheta_{3}c$.

In \citep{Mifdal:1997} the exponential stability of the closed-loop system is proven for the control law \eqref{control_law_general} with $\vartheta_{4}=0$ under certain further conditions by using an energy multiplier approach. The proof of stability is performed in a rigorous and elegant way, however, there is no systematic design of the control law and it is not clear how to specifically tune the controller parameters. In contrast, the control law \eqref{control_law_general} used in this paper overcomes these deficiencies, but then the controller parameters for tuning appear in all parameters $\vartheta_{j}$, $j=1,\ldots,4$, which will be shown subsequently, and $\vartheta_{4}$ has to be unequal to zero. For the authors, it is not obvious and does not seem to be straightforward to extend the stability proof of \cite{Mifdal:1997} to the case $\vartheta_{4}\neq0$ under the same non-restrictive conditions on the controller parameters as presented in this paper.

For a controller with the same structure as in \eqref{control_law_general} but, compared to the controller considered in this paper, with totally different conditions on the parameters $\vartheta_{j}$, $j=1,\ldots,4$, the asymptotic and exponential stability of the closed-loop system is shown by \cite{Coron:Novel}. In this paper, the dynamics of the payload \eqref{sub:3} was neglected. Moreover, the controller was designed based on a backstepping approach, however, the authors mention in their paper that this boundary feedback law ensures asymptotic stability, but it is not clear if it produces exponential stability. Therefore, they modify the backstepping controller and based on this they are able to show exponential stability of the closed-loop system.

Inspired by \cite{Coron:Novel}, a controller was systematically designed and experimentally validated for the system \eqref{hamilton} of \cite{TWK06}. In order to reveal the connection between the controller parameters for tuning and the parameters $\vartheta_{j}$, $j=1,\ldots,4$ in \eqref{control_law_general}, the control design will be shortly revisited, see \citep{TWK06} for more details. In contrast to the (simple) passive controller presented by \cite{Conrad:Mifdal}, the (damping) controller design by \cite{TWK06} is based on the idea to specifically influence the energy flow between the cart and the chain, which is represented by the collocated variables cart velocity $\partial_{t}w\left(  t,0\right)  =v\left(  t,0\right)  $ and internal force in the pivot bearing carrying the chains $F_{i}=P\left(  0\right)  \partial_{x}w\left(  t,0\right)  $, thus forming an energy port, see right-hand side of Figure \ref{fig1}. The sum of the potential energy of the chain and the kinetic energy of the chain and the
payload according to \eqref{hamilton} reads as
\begin{equation}\label{energy1}
\bar{H}=\frac{1}{2}\int_{0}^{L}P(x)\left(  \partial_{x}w\left(  t,x\right)\right)  ^{2}+\rho v^{2}\left(  t,x\right)  \mathrm{d}x+\frac{1}{2}m_{p}v^{2}\left(  t,L\right)  \text{ .}%
\end{equation}
The change of $\bar{H}$ along a solution of \eqref{hamilton} yields $\frac{\mathrm{d}}{\mathrm{d}t}\bar{H}=-v\left(  t,0\right)  F_{i}$. Thus, if $v\left(t,0\right)$ were the (virtual) control input, the control law $v\left(t,0\right)  :=\chi_{1}F_{i}$, with the controller parameter $\chi_{1}>0$, would render the closed-loop system passive. However, this would ensure a good damping of the chain vibrations, but the cart position $w\left(  t,0\right)$ remains unconsidered within this approach. Therefore, the energy functional $\bar{H}$ from \eqref{energy1} is extended by the potential energy of a virtual linear\footnote{Note that in this paper the virtual spring force $f_s(\cdot)$ in \citep{TWK06} is considered linear.} spring attached to the cart in the form\footnote{Here the equilibrium of the cart position is set to zero, but can, of course, take any other value in the operating range. }
\begin{equation}\label{energy2}
\bar{V}=\chi_{1}\bar{H}+\frac{\chi_{2}}{2}w^{2}\left(  t,0\right)  \text{,}
\end{equation}
with the controller parameters $\chi_{1}$, $\chi_{2}>0$. The change of $\bar{V}$ along a solution of \eqref{hamilton}
\begin{equation}\label{energy3}
\frac{\mathrm{d}}{\mathrm{d}t}\bar{V}=\left(  \chi_{2}w\left(  t,0\right)-\chi_{1}F_{i}\right)  v\left(  t,0\right)  \text{ }%
\end{equation}
immediately shows that the control law $v\left(  t,0\right)  :=-\left(\chi_{2}w(t,0)-\chi_{1}F_{i}\right)$ for the virtual control input $v\left(  t,0\right)$ makes the closed-loop system passive again. The controller parameters $\chi_{1}$ and $\chi_{2}$ facilitate a simple tuning of the closed-loop behavior, because a larger $\chi_{1}$ (wrt $\chi_{2}$) brings along a higher damping of the chain vibrations at the cost of larger deviations of the cart position $w\left(  t,0\right)  $ from the equilibrium and for a larger $\chi_{2}$ (wrt $\chi_{1}$) the control of the cart position is becoming more important. Since $v\left(  t,0\right)  $ is not the real control input, a simple backstepping approach, see, e.g., \citep{refKrstic}, is applied to \eqref{sub:2}. In a nutshell, the functional $\bar{V}$ from \eqref{energy2} is extended in the form
\begin{equation}\label{energy4}
V=\bar{V}+\frac{1}{2}\left(  v\left(  t,0\right)  +\chi_{2}w\left(t,0\right)  -\chi_{1}F_{i}\right)  ^{2}
\end{equation}
and the control law is designed in such way that
\begin{equation}\label{energy5}
\frac{\mathrm{d}}{\mathrm{d}t}V=-v^{2}\left(  t,0\right)  -\chi_{3}\left(v\left(  t,0\right)  +\chi_{2}w\left(  t,0\right)  -\chi_{1}F_{i}\right)^{2}\text{, }%
\end{equation}
with the controller parameter $\chi_{3}>0$. The control law finally takes the form (see also \eqref{control_law_general})
\begin{align}\label{controllerA}
F\left(  t\right)  = & \underset{\vartheta_{1}}{\underbrace{-\left(  \chi_{3}+\chi_{2}+1\right)  m_{c}}}v\left(  t,0\right)  +\underset{\vartheta_{2}}{\underbrace{\chi_{1}P(0)m_{c}}}\partial_{x}v\left(  t,0\right)%
+\underset{\vartheta_{3}}{\underbrace{\left(  -\chi_{3}\chi_{2}m_{c}\right)  }}w\left(  t,0\right) \\
 & +\underset{\vartheta_{4}}{\underbrace{\left(  \chi_{3}\chi_{1}m_{c}-1\right)  P(0)}}\partial_{x}w\left(t,0\right)  \text{ .} \nonumber
\end{align}
The third controller parameter $\chi_{3}$ weights the deviation of the virtual control input $v\left(  t,0\right)  $ from its desired time evolution $\chi_{2}w\left(  t,0\right)  -\chi_{1}F_{i}$. Note that in addition to the feedback control law \eqref{controllerA} the control concept of \cite{TWK06} consists of a flatness-based feedforward controller as presented by \cite{Petit:Rouchon}. Due to the linearity of the system \eqref{hamilton} the trajectory error dynamics are identical and thus also the stability proof remains the same.

\medskip

In \citep{TWK06}, energy dissipation of the closed-loop system was shown for \eqref{controllerA}. \cite{Thull:Wild:Kugi:at} attempted to show asymptotic stability by using LaSalle's invariance principle (see, \citep{lgm}). This is common practice in the context of (hyperbolic) control systems, see, e.g., ~\citep{msa_15_1, MSAK, Chentouf:Couchouron, conrad1998stabilization,Coron:Novel, KT05,  Morgul2001} for the control of an Euler-Bernoulli beam, and \citep{MR1173433, Morgul1994, MR1828905} for the control of hanging cables. However, with the (energy) inner product chosen by \cite{Thull:Wild:Kugi:at,TWK06} the proof of the closed-loop stability did not work out, which was also correctly pointed out by \cite{grab1}. The backstepping approach presented by \cite{TWK06} is quite intuitive from a control point of view, but it brings along the drawback that the control law depends on $\pd_xv(t,0)$, see \eqref{controllerA}, which makes it impossible to analyze the closed-loop system in the space $H^1\times L^2$.
Therefore, an appropriate Hilbert space $\H$ and a convenient inner product, see \eqref{inner_prod}, is introduced in this paper which allows the application of the Lumer-Phillips theorem and the rigorous proof of the closed-loop stability of \eqref{hamilton} with \eqref{controllerA}. Due to the change of the abstract model setting, the strategy to prove asymptotic stability does not make use of LaSalle's invariance principle. Instead, techniques from spectral analysis are used, see \citep{lgm}. In order to show exponential stability, Huang's theorem, see, e.g.,~\citep{ref:huang}, is employed.

In order to prove stability of the closed-loop system, in a first step \eqref{hamilton} with \eqref{controllerA} is rewritten as an abstract evolution equation $\dot y= Ay$ in an appropriate Hilbert space $\H$, where the generator $A$ is a linear operator, see, e.g., \citep{MR1828905,grab2,grab1,Morgul1994,lgm} for the formulation of similar systems describing hanging cables. We start by showing that $A$ generates a $C_0$-semigroup of contractions, by using the Lumer-Phillips theorem. To this end an inner product, equivalent to the natural inner product in $\H$, is used. Then, we show that the inverse $A^{-1}$ exists and is compact. This implies that the spectrum $\sigma(A)$ consists entirely of eigenvalues. Since $A$ generates a $C_0$-semigroup of uniformly bounded operators, the Hille-Yosida theorem implies that $\sigma(A)\subseteq \{\zeta\in \C:\Re\zeta\le 0\}$. We then show that $\i\R$ lies in $\rho(A)$, by demonstrating that for all $\lambda\in\R$ the eigenvalue equation $Ay=\i\lambda y$ only has the
trivial solution. According to \citep[][Theorem~3.26]{lgm} this proves the asymptotic stability of the system. Finally, we show uniform boundedness of the resolvent $(\i\lambda-A)^{-1}$ for $\lambda\in\R$. Huang's Theorem (cf.~Corollary 3.36 of \cite{lgm}, see also \citep{ref:huang}) then implies exponential stability of the closed-loop system.

\medskip

The paper is organized as follows: In Section \ref{sec:2}, we prove that $A$ generates a $C_0$-semigroup of uniformly bounded operators. In Section \ref{sec:3}, we show the asymptotic stability of this semigroup, and  Section \ref{sec:4} is devoted to the proof of the exponential stability. Finally, Section \ref{sec:5} contain some conclusions.

%%%%%%%%%%%%%%%%%%%%%%%%%%%%%%%%%%%%%%%%%%%%%%%%%%%%%%%%%%%%%%%%%%%%%%%%%%%%%%

\section{Formulation as a Dissipative Evolution Equation}\label{sec:2}

For the mathematical analysis of the system \eqref{hamilton} with \eqref{controllerA} it is convenient to eliminate most numerical coefficients. To this end, we rescale length and time, i.e.~we introduce new variables $\tilde x= \frac{P(L)\rho}{m_p}x$ and $\tilde t = \frac{P(L)}{m_p}\sqrt{\rho}t$. With $\tilde w(\tilde t,\tilde x) := w(t,x)$ and $\tilde P(\tilde x) = P(x)$ the system \eqref{hamilton} is equivalent to
\begin{align}
	\pd_{\tilde t}^2\tilde w(\tilde t,\tilde x)&=\pd_{\tilde x}(\tilde P(\tilde x)\pd_{\tilde x}\tilde w(\tilde t,\tilde x),\quad \tilde x\in(0,\tilde L), \tilde t>0,\tag{\ref{sub:1}'}\\
	\pd_{\tilde t}^2\tilde w(\tilde t,\tilde L)&=-\pd_{\tilde x}\tilde w(\tilde t,\tilde L),\tag{\ref{sub:3}'}\\
	\pd_{\tilde t}^2\tilde  w(\tilde t,0)&=\tilde\theta_1\tilde v(\tilde t,0)+\tilde\theta_2\pd_{\tilde x}\tilde v(\tilde t,0)+\tilde\theta_3\tilde w(\tilde t,0)+\tilde\theta_4\pd_{\tilde x}\tilde w(\tilde t,0),\tag{\ref{sub:2}'}
\end{align}
in new coordinates. In (\ref{sub:2}'), $m_c$ and all additional factors arising from the change of coordinates as well as the term $\tilde P(0)\pd_{\tilde x} \tilde w(\tilde t,0)$ have been merged in the new coefficients $\tilde \theta_i$, $i=1,\ldots,4$. In the following, we only consider the system  (\ref{hamilton}'). However, for the sake of readability, we will omit the superscript tilde in the sequel and simply write $x,\,t$, $\theta_i$, $w$, and $P$.

For the analysis of (\ref{hamilton}') we define the (complex) Hilbert space
\begin{equation}\label{1}
\H=\{z=(w,v,\xi,\psi):w\in H^2(0,L),\,v\in H^1(0,L),\, \xi= v(L),\,\psi=v(0)\},
\end{equation}
which is a closed subspace of $H^2\times H^1\times \C\times\C$. Here, $H^n(0,L)$ denotes the Sobolev space of functions whose derivatives up to order $n$ are square-integrable \citep[see][for details]{ad}. The auxiliary scalar variables $\xi$, $\psi$ are introduced here in order to include the dynamical boundary conditions (\ref{sub:3}') and (\ref{sub:2}') into the initial value problem. $\H$ is equipped with the natural inner product
\begin{equation}\la z_1, z_2\ra=\la w_1, w_2\ra_{H^2}+\la v_1,v_2\ra_{H^1}+\xi_1\bar\xi_2+\psi_1\bar\psi_2,\end{equation}
where $\bar \xi$ denotes the complex conjugate of $\xi$. Let the linear operator $A:D(A)\subset \H\to\H$ be defined as
\begin{equation}\label{de:a}
A:\left[\begin{array}{c}w\\v\\\xi\\\psi\end{array}\right]\mapsto \left[\begin{array}{c}v\\(P w')'\\-
w'(L)\\\theta_1v(0)+\theta_2v'(0)+\theta_3w(0)+\theta_4 w'(0)\end{array}\right],
\end{equation}
where $w'$ denotes the spatial derivative of $w$, i.e.~$w'=\pd_xw$. The (dense) domain of $A$ is defined as
\begin{align}
	D(A):=\big\{z=(w,v,\xi,\psi)&:w\in H^3(0,L),\, v\in H^2(0,L),\,\xi=v(L),\, \psi=v(0), \label{d_a}\\
		&\quad(P w')'(L)=-w'(L), \,(Pw')'(0)=F[w,v]\big\},\nonumber
\end{align}
{\color{black} with
$$
  F[w,v]:=\theta_1 v(0) +\theta_2 v'(0) +\theta_3 w(0) +\theta_4 w'(0)
$$
due to (\ref{hamilton}').}

The boundary conditions stated in $D(A)$ arise naturally from the requirement that $\ran A\subset \H$. With these definitions, we can rewrite the system (\ref{hamilton}') as the following initial value problem in $\H$:
\begin{equation}\label{ivp:0}
\begin{cases}
\dot z(t)=Az(t),\\
z(0)=z_0\in \H.
\end{cases}
\end{equation}

For some of the following proofs, the natural inner product $\la\cdot,\cdot\ra$ on $\H$ is unpractical. Therefore, we define an equivalent inner product, which is more suitable for the considered problem:
\begin{align}
\la z_1, z_2\rah &:=\alpha_1\int_0^L\left[ {\gamma}(P w_1')'(P\bar w_2')'+P w'_1 \bar w_2'\right]\d x+\alpha_1\gamma
P(L) w'_1(L)  \bar w_2'(L)\label{inner_prod}\\
			&\quad+{\alpha_2}w_1(0)  \bar w_2(0)+\alpha_1\int_0^L\left(\gamma P v_1' \bar  v_2'+ v_1  \bar
v_2\right)\d x+\alpha_1P(L)\xi_1 \bar \xi_2+{\alpha_2\gamma}\psi_1 \bar \psi_2\nonumber\\
			&\quad+\frac12\big(\psi_1-2\alpha_1P(0) w_1'(0)+2\alpha_2w_1(0)\big)\big( \bar \psi_2-2\alpha_1P(0)
 \bar w_2'(0)+2\alpha_2  \bar w_2(0)\big),\nonumber
\end{align}
where $\alpha_1,\alpha_2$, and $\gamma$ are positive constants to be specified later (in Lemma \ref{lem:diss} and the corresponding proof). We have the following lemma:

\begin{lem}\label{equivalence}
The norm $\|\cdot\|_\H$ is equivalent to the natural norm $\|\cdot\|$ on $\H$.
\end{lem}

\begin{proof}
We have to prove the existence of constants $c_1,c_2>0$ such that $c_1\|z\|\le \|z\|_\H\le c_2\|z\|$ holds for all $z\in\H$. To verify the first inequality, it remains to show the existence of $\tilde c_1$ such that
\begin{equation}\label{ineq:1}
\int_0^L\Big[{\gamma}|(P w')'|^2+P| w'|^2\Big]\d x\ge \tilde c_1\int_0^L\Big[| w''|^2+| w'|^2\Big]\d x
\end{equation}
holds for all real-valued $w\in H^2(0,L)$. Using the properties of $P$ mentioned above, Lemma \ref{lem:A:1} (see Appendix \ref{app:a}) can be applied pointwise in $x$ with $a=\sqrt{{\gamma}}\, P'(x),\,b=\sqrt{{\gamma}}\, P(x),\,\epsilon=P(x),\, x_1= |w'(x)|$, and $x_2= |w''(x)|$, which directly yields the  desired inequality (\ref{ineq:1}).

To verify the second inequality, it suffices to apply Cauchy's inequality $ab\le \frac{a^2}2+\frac{b^2}2, \, a,b\in\R$, to the terms obtained by expansion of the last term in $\|z\|^2_\H$.
\end{proof}

The main statement of this section is the following theorem, which will be proved in several steps:

\begin{trm}\label{trm:1}
Let there be constants $a,b>0$ satisfying $(a+b-1)^2<4ab$, such that
\begin{equation}\label{cond:theta}
	\theta_1=\frac{\theta_3}b-a,\quad	\theta_2=\frac{\theta_4}b,
\end{equation}
and $\theta_1,\theta_3<0$ and $\theta_2,\theta_4>0$. Then, the operator $A$ is the infinitesimal generator of a $C_0$-semigroup of uniformly bounded operators $\{T(t)\}_{t\ge0}$ on $\H$.
\end{trm}

The conditions of Theorem \ref{trm:1} on the parameters $\vartheta_{j}$, $j=1,\ldots,4$ from (\ref{sub:2}') are fulfilled if the controller parameters $\chi_{1}$, $\chi_{2}$, and $\chi_{3}$ according to \eqref{controllerA} meet the following non-restrictive inequality constraints
\begin{equation}\label{conditionI}
\chi_{1}>0\text{,\quad }\chi_{2}>0\text{,\quad and }\chi_{3}>\frac{\left(m_{p}-P\left(  L\right)  \sqrt{\rho}\right)  ^{2}}{4m_{p}P\left(  L\right)\sqrt{\rho}}\text{, }%
\end{equation}
where
\begin{equation}\label{conditionII}
b=\frac{\tilde{\vartheta}_{4}}{\tilde{\vartheta}_{2}}=\frac{m_{p}\alpha_{3}}{P\left(  L\right)  \sqrt{\rho}}>0\text{\quad and \quad}a=\frac{\tilde{\vartheta}_{3}}{b}-\tilde{\vartheta}_{1}=\frac{\left(  \alpha_{3}+1\right)m_{p}}{P\left(L\right)  \sqrt{\rho}}>0\text{ .}%
\end{equation}
{\color{black} In Theorem \ref{trm:1}, the admissible parameters $(a,b)$ lie inside a parabola in the first quadrant, and this parabola is tangent to the positive $a-$ and $b-$axes.}

{\color{black}
\begin{rem}
 Let us briefly compare the model (\ref{hamilton}') subject to \eqref{cond:theta} with the closed-loop system from \citep{Mifdal:1997} and even its extension in \S0(a) of \citep{Mifdal-diss:1997}. We recall that the latter two models correspond to a control law with $\theta_4=0$. (\ref{hamilton}'), \eqref{cond:theta} can only be matched with the models analyzed by Mifdal when using the special parameters $a=0$ and $b=\frac{P(0)}{P(L)} \frac{M}{\theta_2 m}$ in \eqref{cond:theta}. Here, $M$ and $m$ denote parameters related to the masses of the payload and the cart in \citep{Mifdal:1997}. Since Theorem \ref{trm:1} requires $a>0$, these two types of models and results are complementary to each other.
\end{rem}}

\begin{rem}
 With respect to the inner product specified in Lemma \ref{lem:diss}, this semigroup is even a semigroup of contractions, see the proof of Theorem \ref{trm:1} below.
\end{rem}

We shall prove Theorem~\ref{trm:1} by applying the Lumer-Phillips theorem \citep[cf.][]{pazy}. But before, we verify some basic properties of $A$.

\begin{lem}\label{lem:dense}
The domain $D(A)$ defined in (\ref{de:a}) is dense in $\H$.
\end{lem}

\begin{proof}
Let $z_0=(w_0,v_0,\xi_0,\psi_0)\in \H$. Since the inclusions $H^3(0,L)\subset H^2(0,L)\subset H^1(0,L)$ are dense, there
exists a sequence $z_n=(w_n,v_n,\xi_n,\psi_n)\in H^3(0,L)\times H^2(0,L)\times \C^2\cap \H$ such that $z_n\to z_0$ in
$\H$. % This \H produces an error message on Anton's laptop, but output is correct.
Now, in general, the second derivatives $\pd^2w_n(0),\pd^2w_n(L)$ will not satisfy the boundary conditions
necessary for $z_n\in D(A)$.\

The fact that $H_0^1(0,L)\subset L^2(0,L)$ is dense ensures the existence of a sequence $\{u_n\}\subset H^1(0,L)$ satisfying $u_n(0)=a$ for all $n\in\N$ and any fixed $a\in\C$, with $\|u_n\|_{L^2}\to 0$. The sequence $\{y_n\}$ defined by
\[y_n:=\int_0^x\int_0^\xi u_n(\zeta)\d \zeta\d \xi\]
satisfies $\pd ^2 y_n(0)=a$ for all $n\in\N$, and $\|y_n\|_{H^2}\to 0$.

This shows that, for the sequence $\{w_n\}$, the
values $\pd^2w_n(0),\,\pd^2w_n(L)$ can be modified such that the modified sequence $\{\tilde z_n\}\subset D(A)$, but still $\tilde z_n\to z_0$ in $\H$.
\end{proof}

\begin{lem}\label{lem:inv}
Under the condition $\theta_3\neq 0$, the operator $A$ is injective and $\ran A=\H$, i.e.~$A^{-1}$ exists and $D(A^{-1})=\H$.
\end{lem}

\begin{proof}
We prove this lemma by showing that the equation $Az=(f,g,g(L),g(0))$ has a unique solution $z\in D(A)$ for every
$(f,g,g(L),g(0))\in\H.$ This equation reads in detail:
\begin{equation}\label{inv:1}
\left[\begin{array}{c}v\\(Pw')'\\- w'(L)\\\theta_1v(0)+\theta_2 v'(0)+\theta_3w(0)+\theta_4 w'(0)\end{array}\right]
=\left[\begin{array}{c}f\\g\\g(L)\\g(0)\end{array}\right].
\end{equation}
{}From the first line we immediately find $v=f\in H^2(0,L)$, which also fixes the values $v(0)$ and $ v'(0)$. After integration of the second line we obtain
\begin{equation}\label{dw}
 w'(x)=-\frac{P(L)}{P(x)} g(L)+\frac{1}{P(x)}\int_L^xg(y)\d y,
\end{equation}
where we used $w'(L)=-g(L)$ from the third line. Since $1/P\in H^2(0,L)$ and $g\in H^1(0,L)$, we find $w'\in H^2(0,L)$.
This equation also determines $w'(0)$. In combination with the already known values $v(0)$ and $ v'(0)$, we obtain $w(0)$ from the fourth line in (\ref{inv:1}), since $\theta_3\neq 0$. Hence, $w(x)$ is uniquely determined as:
\begin{equation}\label{w}
w(x)=w(0)-\int_0^x\frac{P(L)}{P(y)} g(L)\d y+\int_0^x\frac{1}{P(y)}\int_L^yg(\zeta)\d \zeta\d y.
\end{equation}
All integrals exist, since $P(x)>0$ holds uniformly. Finally, $w\in H^3(0,L)$ holds. Thus, the
inverse $A^{-1}$ exists and is defined on $\H$.
\end{proof}

\begin{lem}\label{lem:compact}
If $\theta_3\neq 0$, the operator $A^{-1}$ is compact.
\end{lem}

\begin{proof}
We show that for $(f,g,g(L),g(0))\in\H$ the norm of $z=A^{-1}(f,g,g(L),g(0))$ in $\mathcal J:=H^3(0,L)\times
H^2(0,L)\times \C^2$ is uniformly bounded by $\|(f,g,g(L),g(0))\|_\H$.\\ Due to the continuous embedding
$H^1(0,L)\hookrightarrow C[0,L]$ in one dimension \citep[see, e.g.,][]{ad}, we have the estimates $|g(L)|,\,|g(0)|\le C
\|g\|_{H^1}$. Here and in the sequel, $C$ denotes positive, not necessarily equal constants. From the third line in (\ref{inv:1}) we therefore get $| w'(L)|\le C \|g\|_{H^1}$. With this and
(\ref{dw}) we find the estimate
\begin{equation}\label{absdw}
\| w'\|_{L^2}\le C \|g\|_{H^1}.
\end{equation}
Next we will apply this result to the identity $Pw''=g-P' w'$, which is obtained from the second line in (\ref{inv:1}), and use $ P'\in L^\infty(0,L)$ and $P(x)\ge P^0>0$ (with $P^0$ introduced right after \eqref{hamilton}). This yields
\begin{equation}\label{d2w}
\| w''\|_{L^2}\le C \|g\|_{H^1}.
\end{equation}
Similarly, from $(P w')''= g'$ we obtain the estimate
\begin{equation}\label{d3w}
\| w'''\|_{L^2}\le C  \|g\|_{H^1}.
\end{equation}
For $v$ we immediately get $\|v\|_{H^2}=\|f\|_{H^2}$ using the first line in (\ref{inv:1}). Due to the continuous embedding $H^k(0,L)\hookrightarrow C^{k-1}[0,L]$ in one dimension \citep[cf.][]{ad}, we find the following estimates
\begin{align}
	|v(0)|&\le  C \|f\|_{H^2},\label{v0}\\
	| v'(0)|&\le C \|f\|_{H^2}.\label{dv0}
\end{align}
Using the above estimate for $ w'(L)$ and (\ref{dw}) we obtain
\begin{equation}\label{dw0}
| w'(0)|\le C \|g\|_{H^1}.
\end{equation}
Applying (\ref{v0}), (\ref{dv0}), (\ref{dw0}) to the fourth line of (\ref{inv:1}) and using $|g(0)|\le C\|g\|_{H^1}$ yields
\begin{equation}\label{w0}|w(0)|^2\le C (\|f\|^2_{H^2}+\|g\|^2_{H^1}).\end{equation}
Altogether, we get
\begin{equation}\label{estw}
\|w\|_{H^3}^2\le C (\|f\|^2_{H^2}+\|g\|^2_{H^1}).
\end{equation}
Thus, we have $\|w\|_{H^3}^2+\|v\|^2_{H^2}\le C(\|f\|^2_{H^2}+\|g\|^2_{H^1})$, which shows that $A^{-1}$ maps bounded sets in $\H$ into bounded sets in $\mathcal J$. Since the embeddings $H^3(0,L)\subset\subset H^2(0,L)\subset\subset H^1(0,L)$ are compact, $A^{-1}$ is a compact operator.
\end{proof}

From the previous lemma we know that $A^{-1}$ is a closed operator, therefore we have:
\begin{cor}\label{cor:1}
For $\theta_3\neq 0$, the operator $A$ is closed, and $0\in\rho(A)$, the resolvent set of $A$.
\end{cor}

Now we turn to the application of the Lumer-Phillips theorem in order to prove Theorem \ref{trm:1}. To this end we shall prove the dissipativity of $A$ with respect to the inner product $\la \cdot,\cdot\ra_\H$.

\begin{lem}\label{lem:diss}
Let the assumptions of Theorem \ref{trm:1} hold, and let $\H$ be equipped with the inner product (\ref{inner_prod}), where we choose the coefficients
\begin{equation}\label{alphas}
  \alpha_1:=\frac{\theta_2}{2P(0)},\qquad \alpha_2:=-\frac{\theta_2\theta_3}{2\theta_4},
\end{equation}
and $\gamma>0$ is sufficiently small. Then the operator $A$ is dissipative in $\H$.
\end{lem}

The proof is deferred to the Appendix~\ref{app:c}. Now, Theorem \ref{trm:1} follows directly from the above results:

\begin{proof}[Proof of Theorem \ref{trm:1}]
First we prove this result under the additional assumptions of Lemma \ref{lem:diss} (on $\alpha_1$, $\alpha_2$ and $\gamma$). Then $A$ is dissipative in $\H$ equipped with $\|\cdot\|_\H$, and Corollary \ref{cor:1} implies $0\in\rho(A)$. Since $\rho(A)$ is an open set, there exists some $\zeta\in\rho(A)$ with positive real part. So the requirements of the Lumer-Phillips theorem are fulfilled, and we obtain that $A$ generates a $C_0$-semigroup of contractions on $\H$ with respect to $\|\cdot\|_\H$.

Next we drop the additional assumptions of Lemma~\ref{lem:diss}, and consider an equivalent norm on $\H$. Then the semigroup $\{T(t)\}_{t\ge 0}$ is not necessarily a contraction semigroup any more, but still a $C_0$-semigroup of uniformly bounded operators.
\end{proof}

The following corollary follows as a consequence of Theorem \ref{trm:1}, due to elementary properties of generators of
$C_0$-semigroups of operators \citep[for more details, see][]{pazy}:

\begin{cor}
Under the assumptions of Theorem \ref{trm:1}, the initial value problem
\begin{equation}\label{ivp:1}
\begin{cases}
\dot z(t)=Az(t)\\
z(0)=z_0
\end{cases}
\end{equation}
has a unique mild solution $z(t):=T(t)z_0$ for all $z_0\in\H$, where $\{T(t)\}_{t\ge 0}$ is the $C_0$-semigroup
generated by $A$. If $z_0\in D(A)$, then $z(t)$ is continuously differentiable on $[0,\infty)$ and $z(t)\in D(A)$ for
all $t\ge 0$, and therefore is a classical solution. Furthermore, the norm $\|z(t)\|_\H$ remains bounded as $t\to\infty$.
\end{cor}

%\begin{rem}\label{compare}
%The particular controller developed in \cite{TWK06} satisfies the conditions for $\theta_1,\ldots,\theta_4$ in Theorem
%\ref{trm:1} with $b=a+1$. With this identity, the condition $(a+b-1)^2<4ab$ clearly holds, and therefore the controller
%generates a $C_0$-semigroup of uniformly bounded operators.
%\end{rem}

%%%%%%%%%%%%%%%%%%%%%%%%%%%%%%%%%%%%%%%%%%%%%%%%%%%%%%%%%%%%%%%%%%%%%%%%%%%%%%%%%%%

\bigskip

\section{Asymptotic Stability}\label{sec:3}

After having shown that the norm of every solution of the initial value problem (\ref{ivp:1}) is uniformly bounded with respect to $t\ge0$, we now prove that the norm even tends to zero as $t\to\infty$, i.e.~the $C_0$-semigroup $\{T(t)\}_{t\ge 0}$ generated by $A$ is asymptotically stable, by applying the following theorem \cite[see][Theorem 3.26]{lgm}:

\begin{trm}\label{trm:3.26}
Let $\{S(t)\}_{t\ge 0}$ be a uniformly bounded $C_0$-semigroup in a Banach space $X$ with generator $\mathcal A$, and
assume that the resolvent $R(\lambda,\mathcal A)$ is compact for some $\lambda\in\rho(\mathcal A)$. Then $\{S(t)\}_{t\ge
0}$ is asymptotically stable if and only if $\Re \lambda <0$ for all $\lambda\in\sigma(\mathcal A)$.
\end{trm}

\begin{rem}
The compactness of the resolvent $R(\lambda,\mathcal A)$ for one $\lambda\in\rho(\mathcal A)$ already implies its compactness for all $\lambda\in\rho(\mathcal A)$, \cite[cf.][Theorem~III.6.29]{kato}.
\end{rem}

\begin{trm}\label{lem:3.3}
Let the assumptions of Theorem \ref{trm:1} hold. Then the $C_0$-semigroup $\{T(t)\}_{t\ge 0}$ generated by $A$ is asymptotically stable.
\end{trm}

\begin{proof}
According to Lemma \ref{lem:compact} the operator $A$ has compact resolvent, and the associated semigroup $\{T(t)\}_{t\ge 0}$ is uniformly bounded due to Theorem \ref{trm:1}. As a consequence of the Hille-Yosida theorem \cite[see][Corollary~1.3.6]{pazy}, this implies $\sigma(A)\subseteq\{ \lambda\in \C:\Re \lambda\le 0\}$. Hence, in order to apply Theorem \ref{trm:3.26}, it remains to prove that $\i\R\subset \rho (A)$. Since the resolvent is compact, $\sigma(A)$ consists only of eigenvalues. Thus, it is sufficient to show that $A-\i\tau$ is injective for all $\tau \in \R$, that is to show that the system
\begin{equation}\label{inj:1}
\left[\begin{array}{c}v-\i\tau w\\(Pw')'-\i\tau v\\- w'(L)-\i \tau v(L)\\\theta_1v(0)+\theta_2
v'(0)+\theta_3w(0)+\theta_4 w'(0)-\i\tau v(0)\end{array}\right]=0
\end{equation}
only has the trivial solution in $D(A)$. We can rewrite this system in terms of the following equivalent boundary value problem for $w\in H^3(0,L)\hookrightarrow C^2[0,L]:$
\begin{subequations}\label{123r}
\begin{align}
	 (P w')'+ \tau ^2w&=0,\quad\qquad x\in(0,L),\label{bvp:1}\\
	 w'(0)&=c_0w(0),\label{bc:1}\\
	 w'(L)&=c_Lw(L),\label{bc:2}
\end{align}
\end{subequations}
where $c_0:=-\frac{\theta_3+\tau^2+\i\tau \theta_1}{\theta_4+\i\tau\theta_2} $ and $c_L:=\tau^2$. It is important to note that the conditions (\ref{cond:theta}) on the $\theta_i$ imply that $c_0\notin\R$ for all $\tau \in\R$. We now multiply equation (\ref{bvp:1}) by the complex conjugate $\bar w$ and integrate by parts, which yields
\[-\int_0^LP| w'|^2\d x+\tau^2\|w\|^2_{L^2}+P(L) w'(L)\bar w(L)=P(0) w'(0)\bar w(0).\]
Due to the boundary conditions (\ref{bc:1}) and (\ref{bc:2}) the left hand side of above identity is real, but the right hand side is either non-real or zero. Thus $ w'(0)\bar w(0)=0$, and (\ref{bc:1}) implies that $w(0)= w'(0)=0$. Therefore, every solution of the boundary value problem (\ref{123r}) also satisfies the initial value problem
\begin{align*}
	 (P w')'+ \tau ^2w&=0,\qquad x\in(0,L),\\
    w(0)&=0,\\
	 w'(0)&=0.
\end{align*}
Hence, $w\equiv0$, and this shows that $A-\i\tau$ is injective for all $\tau\in\R$.
\end{proof}

%%%%%%%%%%%%%%%%%%%%%%%%%%%%%%%%%%%%%%%%%%%%%%%%%%%%%%%%%%%%%%%%%%%%%%%%%%%%%%%%%%%%%%%%%%%%%%%

\bigskip

\section{Exponential Stability}\label{sec:4}

Here we show an even stronger result, namely the exponential stability of the semigroup $\{T(t)\}_{t\ge 0}$, i.e.~we prove that every solution of the initial value problem (\ref{ivp:1}) tends to zero exponentially. We follow a strategy similar to the one applied by \cite{Morgul2001}.

\begin{de}[Exponential stability]
A $C_0$-semigroup $\{S(t)\}_{t\ge0}$ is said to be {\em exponentially stable} if there exist constants $M\ge1$ and
$\omega>0$ such that $\|S(t)\|\le M\exp(-\omega t)$ for all $t\ge 0$.
\end{de}

To investigate exponential stability of a $C_0$-semigroup, we use the following theorem \cite[see][Corollary~3.36]{lgm}:

\begin{trm}[Huang]\label{huang}
Let $\{S(t)\}_{t\ge0}$ be a uniformly bounded $C_0$-semigroup in a Hilbert space, and let $\mathcal A$ be its generator. Then $\{S(t)\}_{t\ge0}$ is exponentially stable if and only if  $\i\R\subset \rho(\mathcal A)$ and
\begin{equation}\label{est:huang}
\sup_{\tau\in\R}\|R(\i\tau,\mathcal A)\|<\infty.
\end{equation}
\end{trm}

\begin{trm}\label{trm:exp}
Assume that the conditions in Theorem \ref{trm:1} are satisfied. Then the $C_0$-semigroup $\{T(t)\}_{t\ge 0}$ generated by $A$ is exponentially stable.
\end{trm}

\begin{proof}[Proof]

We know from Theorem \ref{lem:3.3} that $\{T(t)\}_{t\ge 0}$ is asymptotically stable, and that $\i\R\subset\rho(A)$. The map $\lambda\mapsto R(\lambda,A)$ is analytic on $\rho(A)$ \cite[cf.][]{yosida}, so, in particular, $\lambda\mapsto \|R(\lambda,A)\|$ is continuous on $\i\R$. In order to apply Theorem \ref{huang} it therefore remains to prove that $\|R(\i\tau,A)\|$ is uniformly bounded as $|\tau|\to \infty$. To this end we need to find a $\tau$-uniform estimate for the solution $z=(w,v,v(L),v(0))$ of the equation
\begin{equation}\label{resolv:eq}(A-\i\tau)z=(f,g, g(L),g(0))\in\H\end{equation}
in terms of the right hand side. The corresponding homogeneous problem \eqref{inj:1} only has the trivial solution (cf.~the proof of Theorem~\ref{lem:3.3}). Hence, we show that the unique solution $(w,v)$ of the BVP
\begin{subequations}\label{12346t}
\begin{align}
	v-\i\tau w &= f,\qquad x\in(0,L),\label{bvp:res:1}\\
	(Pw')' -\i\tau v  &= g ,\qquad x\in(0,L),\label{bvp:res:2}\\
	- w'(L)-\i\tau v(L) &= g(L),\label{bvp:res:3}\\
	\theta_1 v(0)+\theta_2 v'(0)+\theta_3w(0)+\theta_4 w'(0)-\i\tau  v(0) &= g(0)\label{bvp:res:4}
\end{align}
\end{subequations}
satisfies the estimate
\begin{equation}\label{est:res:1}
\|w\|_{H^2}+\|v\|_{H^1}\le C(\|f\|_{H^2}+\|g\|_{H^1})
\end{equation}
uniformly for all $f\in H^2(0,L),\,g \in H^1(0,L)$ and for all $|\tau|$ sufficiently large. \\
Since $v$ and $w$ are directly related via equation (\ref{bvp:res:1}), we replace $v$ in
(\ref{bvp:res:2})-(\ref{bvp:res:4}) by $v=f+\i\tau w$ to obtain the following BVP for $w:$
\begin{subequations}\label{1234s}
\begin{align}
	(Pw')' +\tau^2 w &= g +\i\tau f,\qquad x\in(0,L),\label{bvp:w:1}\\
	- w'(L)+\tau^2 w(L) &= (g+\i\tau f)(L),\label{bvp:w:2}\\
	\underbrace{(\theta_4+\i\tau\theta_2)}_{=:\gamma_1}
w'(0)+\underbrace{(\theta_3+\tau^2+\i\tau\theta_1)}_{=:\gamma_2}w(0)&=(g+\i\tau f)(0)-\theta_1 f(0)-\theta_2
f'(0).\label{bvp:w:3}
\end{align}
\end{subequations}
With this, we first show the desired estimate for $w$.\\

{\bf Step 1: Homogeneous boundary conditions.}\\
To begin with, we shall transform (\ref{1234s}) into a BVP with homogeneous boundary conditions. To this end, we use
(\ref{bvp:w:1}) to eliminate the terms $w(0)$ and $w(L)$. This yields, after differentiating (\ref{bvp:w:1}), the
following BVP for $\tilde y:=P w':$
\begin{subequations}\label{1234sy}
\begin{align}
	\tilde y''+\frac{\tau^2}{P}\tilde y&=g'+\i\tau  f',\qquad\qquad x\in (0,L),\label{bvp:y:1}\\
	 \tilde y(L)+ P(L)\tilde y'(L) &=0,\label{bvp:y:2}\\
	\frac{\gamma_1}{P(0)}\tilde y(0)-\frac{\gamma_2}{\tau^2} \tilde y'(0) &=
\underbrace{-\frac{g(0)}{\tau^2}(\theta_3+\i\tau\theta_1)-\frac{\i\theta_3}\tau f(0)-\theta_2
f'(0)}_{=:R_1}.\label{bvp:y:3}
\end{align}
\end{subequations}
In order to make the second boundary condition homogeneous, we determine a first order polynomial $h(x)=a_1x+a_0$, such
that $h(x)$ satisfies the boundary conditions (\ref{bvp:y:2}),(\ref{bvp:y:3}). The coefficients can be determined
uniquely:
\begin{equation}\label{k_d}
	a_1=-\frac{\tau^2 P(0)R_1}{\gamma_1\tau^2(L+{P(L)})+P(0)\gamma_2},\qquad a_0=-(L+{P(L)})a_1.
\end{equation}
We note that, as already mentioned in the proof of Theorem~\ref{lem:3.3}, $\gamma_1/\gamma_2=1/c_0\notin\R$, and so $a_1$ is always well
defined. For $|\tau|>1$ we find the estimate
\begin{equation}\label{est:k}|a_j|\le \frac C{\tau^2}(\|g\|_{H^1}+|\tau|\|f\|_{H^2}),\quad \text{for } j=0,1,\end{equation}
by using the continuous embedding $H^k(0,L)\hookrightarrow C^{k-1}[0,L]$ in one dimension \cite[cf.][]{ad} to estimate
the terms occurring in $R_1$. Now, the function $y:=\tilde y-h$ satisfies the following problem with homogeneous boundary
conditions:
\begin{subequations}\label{12s4}
\begin{align}
	y''+\frac{\tau^2}{P} y&=H:= g'+\i\tau  f'-\frac{\tau^2}{P}h,\qquad x\in (0,L),\label{bvp:yy:1}\\
	 y(L)+ P(L)y'(L) &=0,\label{bvp:yy:2}\\
	\frac{\gamma_1}{P(0)} y(0)-\frac{\gamma_2}{\tau^2} y'(0) &=0.\label{bvp:yy:3}
\end{align}
\end{subequations}
% We note that the function $H$ can be split into two parts $H=H_f+g'$, with $H_f:=\i\tau f'-\frac{\tau^2}{P}h$.

{\bf Step 2: Solution estimate.}\\
Now we determine the solution of (\ref{12s4}). Let $\{\phi_1,\,\phi_2\}$ be a basis of solutions of the homogeneous
equation  $ y''+\frac{\tau^2}{P} y=0$. Then, the general solution of the inhomogeneous equation (\ref{bvp:yy:1}) can be
obtained by variation of constants:
\begin{align}
y(x)&=c_1\phi_1(x)+c_2\phi_2(x)+\int_0^xH(t)\frac
{\phi_1(t)\phi_2(x)-\phi_2(t)\phi_1(x)}{\phi_1(t)\phi'_2(t)-\phi'_1(t)\phi_2(t)}\d t\label{general:y}\\
	&=c_1\phi_1(x)+c_2\phi_2(x)+\int_0^xH(t)J(x,t)\d t,\label{general:y:2}
\end{align}
where $J(x,t)$ is the Green's function introduced in Lemma \ref{lem:J}, and $c_j\in\C$ are arbitrary constants. The
derivative $ y'(x)$ satisfies
\begin{equation}\label{general:dy}
 y'(x)=c_1 \phi'_1(x)+c_2 \phi'_2(x)+\int_0^xH(t)\pd_x J(x,t)\d t.
\end{equation}
In order to determine the constants $c_j$ we now specify the initial conditions of the solutions $\phi_1,\,\phi_2:$
\[\begin{array}{rclcrcl}
	\phi_{1}(0)&=&0,&\quad&\phi_{2}(0)&=&1,\\
	\phi'_{1}(0)&=&\tau,&\quad&\phi'_{2}(0)&=&0.
\end{array}\]
These conditions imply that the functions $\phi_j$ are real-valued. From the boundary conditions
\eqref{bvp:yy:2}-\eqref{bvp:yy:3} we then find
\begin{equation}\label{c_1}
	c_1 = \frac{-\int_0^LH(t)J(L,t)\d t- P(L)\int_0^LH(t)\pd_x J(L,t)\d
t}{\phi_1(L)+{P(L)}\phi'_1(L)+\frac{\gamma_2P(0)}{\gamma_1\tau}[\phi_2(L)+{P(L)}\phi'_2(L)]},\qquad
	c_2 = \frac{\gamma_2P(0)}{\gamma_1\tau} c_1.
\end{equation}
Again, since $\gamma_2/\gamma_1\notin\R$ and $\phi_1,\,\phi_2$ are linearly independent, the coefficients $c_1,\,c_2$ are well defined. Next we estimate these coefficients. First, we find that
\[\lim_{|\tau|\to\infty}\frac{\gamma_2P(0)}{\gamma_1\tau}= -\frac{\i P(0)}{\theta_2}.\]
Therefore, we can find some constant $C>0$, independent of $|\tau|>1$, such that the denominator $N$ of $c_1$ can be estimated as follows:
\begin{align*}
	|N|^2&:=\left|\phi_1(L)+P(L)\phi'_1(L)+\frac{\gamma_2P(0)}{\gamma_1\tau}[\phi_2(L)+P(L)\phi'_2(L)]\right|^2\\
	&\ge C \left(|\phi_1(L)+P(L)\phi'_1(L)|^2+|\phi_2(L)+P(L)\phi'_2(L)|^2\right).
\end{align*}
{}From the initial conditions of $\phi_1,\phi_2$ and Lemma \ref{lem:J} we find that the Wronskian satisfies $\phi'_1(L)\phi_2(L)-\phi_1(L)\phi'_2(L)=\tau$. Since $\|\phi_j\|_{L^\infty}$ is uniformly bounded for all $\tau$ sufficiently large by Lemma \ref{sol:est}, this implies $|\phi'_1(L)|+|\phi'_2(L)|\ge C\tau$, for some constant $C>0$ independent of $\tau$. With this result, we obtain the estimate
\begin{equation}|N|\ge C|\tau|,\end{equation}
for all $|\tau|>1$, and $C$ independent of $\tau$.\\
Now it remains to estimate the integrals occurring in $c_1$ and those in (\ref{general:y:2}) and (\ref{general:dy}). To this end we split these integrals according to $H=(H-g')+g'$. In order to estimate the integrals corresponding to the first term, we apply Theorem \ref{f_integrl} and use the estimates for $h$ found in (\ref{est:k}). For the other integrals we apply H\"older's inequality, and obtain
\[
 \Big\|\int_0^xg'(t)J(x,t)\d t\Big\|_{L^\infty}\le \|g'\|_{L^1}\|J\|_{L^\infty((0,L)^2)}\le \frac{C}{|\tau|}\|g\|_{H^1},
\]
where we used Lemma \ref{sol:est} to estimate $\|J\|_{L^\infty}$. The integrals with $\pd_x J$ instead of $J$ can be estimated analogously. Altogether we obtain

\begin{align}
	\left\|\int_0^xH(t)J(x,t)\d t\right\|_{L^\infty} &\le  \frac
C{|\tau|}(\|g\|_{H^1}+\|f\|_{H^2}),\label{est:int:1}\\
	\left\|\int_0^xH(t)\pd_xJ(x,t)\d t\right\|_{L^\infty} &\le  C(\|g\|_{H^1}+\|f\|_{H^2}),\label{est:int:2}	
\end{align}
for all $|\tau|>1$, with $C>0$ independent of $\tau$. Therefore we conclude that the estimate $|c_j|\le \frac C{|\tau|}(\|g\|_{H^1}+\|f\|_{H^2})$ holds uniformly in $\tau$. Applying these results and the estimates for the basis-functions $\phi_1,\,\phi_2$ found in Lemma \ref{sol:est} to (\ref{general:y:2}) and (\ref{general:dy}), we find that the following estimates hold uniformly for $|\tau|>1:$
\begin{align}
	\|y\|_{L^2}&\le C\|y\|_{L^\infty}\le \frac C{|\tau|}(\|g\|_{H^1}+\|f\|_{H^2}),\label{final:est:y:1}\\
	\|y\|_{H^1}&\le C(\|y'\|_{L^\infty}+\|y\|_{L^\infty})\le C(\|g\|_{H^1}+\|f\|_{H^2}).\label{final:est:y:2}
\end{align}
Using (\ref{est:k}), we see that the same estimates hold for $\tilde y$. Furthermore, by using $\tilde y=P w'$ and the
equation (\ref{bvp:w:1}) to express $w$ in terms of $w'$ and $w''$, we find
\begin{align}
	\|w\|_{H^1}&\le \frac C{|\tau|}(\|g\|_{H^1}+\|f\|_{H^2}),\label{final:est:w:1}\\
	\|w\|_{H^2}&\le C(\|g\|_{H^1}+\|f\|_{H^2}).\label{final:est:w:2}
\end{align}
Finally, from equation (\ref{bvp:res:1}) and by using (\ref{final:est:w:1}) we get the desired estimate
\[\|v\|_{H^1}\le C(\|g\|_{H^1}+\|f\|_{H^2}),\] which completes the proof.
\end{proof}

\bigskip

%%%%%%%%%%%%%%%%%%%%%%%%%%%%%%%%%%%%%%%%%%%%%%%%%%%%%%%%%%%%%%%%%%%%%%%%%%%%%%%%%%%%%%%%%

\section{Conclusions}\label{sec:5}
In \citep{TWK06}, a backstepping-based controller was proposed for the infinite-dimensional model of a gantry crane with heavy chain and payload. This controller shows excellent results, which was also verified experimentally by \cite{TWK06}. In particular the control law was designed in a systematic way, it features to be robust with respect to unmodeled (stick-slip) friction effects, which are always present in real applications, and it can be easily tuned. Though energy dissipation of the closed-loop system could be shown by \cite{TWK06}, the proof of the closed-loop stability did not work out, which was also correctly pointed out by \cite{grab1}. In this paper, a rigorous proof of the asymptotic and exponential stability of the closed-loop system is given. For this, it was necessary to formulate the dynamics of the closed-loop system as an abstract evolution equation in an appropriate Hilbert space which differs from the space $H^1\times L^2$ which is usually used in the context of heavy
chain systems. Moreover, under very mild conditions on the controller parameters, which were explicitly derived, it was proven that the solutions of the closed-loop system are described by an asymptotically stable $C_0$-semigroup of contractions. Finally, by employing Huang's theorem it was even possible to show that under the same conditions the backstepping-based boundary controller renders the closed-loop system exponentially stable.

%%%%%%%%%%%%%%%%%%%%%%%%%%%%%%%%%%%%%%%%%%%%%%%%%%%%%%%%%%%%%%%%%%%%%%%%%%%%%%%%%%%%%%%%%
\appendix
\section{Useful Inequalities}\label{app:a}

\begin{lem}\label{lem:A:1}
Let $a_0, b_0,\epsilon_0>0$ be given. Then there exist positive constants $c,d$ such that
\begin{equation}\label{ap:1} (ax_1+bx_2)^2+\epsilon x_1^2\ge cx_1^2+dx_2^2\end{equation}
holds uniformly for all $x_1,x_2\in\R$ and $|a|\le a_0,\,b\ge b_0$ and $\epsilon\ge\epsilon_0$.
\end{lem}

\begin{proof}
Inequality (\ref{ap:1}) can be rewritten in the equivalent form
\[\left[\begin{array}{c}x_1\\x_2\end{array}\right]^T\left[\begin{array}{cc} a^2+\epsilon-c & ab\\ ab
&b^2-d\end{array}\right]\left[\begin{array}{c}x_1\\x_2\end{array}\right]\ge 0,  \]
where the occurring matrix will be denoted as $M$. Since this inequality has to hold for all $x_1,x_2\in\R$, it is
equivalent to $M$ being positive semi-definite. Applying the Sylvester criterion yields the following conditions:
\begin{align}
	b^2-d&\ge 0,\label{eq:1}\\
	(\epsilon-c)(b^2-d)&\ge a^2d.\label{eq:2}
\end{align}
If $a=0$, we can take $c=\epsilon_0$ and $d=b^2$. Otherwise, we see from the conditions \eqref{eq:1} and (\ref{eq:2}) that $d< b^2$, so
that \eqref{eq:2} can be written as:
\begin{equation}\label{eq:3}
c\le \epsilon-a^2\frac d{b^2-d}.
\end{equation}
Because of the monotonicity of the right hand side we find the estimate
\[\epsilon-a^2\frac d{b^2-d}\ge \epsilon_0-a_0^2\frac d{b_0^2-d}.\] So, for (\ref{eq:3}) to hold, it is sufficient that
$c,d$ satisfy the stricter inequality
\begin{equation}\label{eq:4}c\le\epsilon_0-a_0^2\frac d{b_0^2-d}.\end{equation} For $d$ sufficiently small, the right hand side becomes positive,
and therefore a $c>0$ satisfying (\ref{eq:4}) exists.
\end{proof}

\begin{lem}\label{lem:A:2}
Let $\alpha,\beta,\delta\in\R$ and
\[
P_3(x_1,x_2,x_3):=x_1^2+x_2^2+x_3^2+2\alpha x_1x_2+2\beta x_2x_3+2\delta x_1x_3
\]
be a polynomial. Then the inequality $P_3(x_1,x_2,x_3)\ge 0$ holds for all $x_1,x_2,x_3\in\R$ if and only if the coefficients satisfy the conditions
\begin{align*}
	\alpha^2\le 1,\quad \beta^2\le1,\quad \delta^2\le 1,\\
	\alpha^2+\beta^2+\delta^2\le 1+2\alpha\beta\delta.
\end{align*}
\end{lem}

\begin{proof}
The polynomial can be written as
\[P_3(x_1,x_2,x_3)=\left[\begin{array}{c}x_1\\x_2\\x_3\end{array}\right]^T \left[\begin{array}{ccc}
1&\alpha&\delta\\\alpha&1&\beta\\\delta&\beta&1\end{array}\right]\left[\begin{array}{c}x_1\\x_2\\x_3\end{array}\right],
\]
with $M$ denoting the $3\times 3$ matrix. Now the property $P_3(x_1,x_2,x_3)\ge 0,\,\,\forall x_1,x_2,x_3\in\R$ is equivalent to $M$ being positive semi-definite. Applying the Sylvester criterion to $M$ yields the desired conditions.
\end{proof}

\bigskip

%%%%%%%%%%%%%%%%%%%%%%%%%%%%%%%%%%%%%%%%%%%%%%%%%%%%%%%%%%%%%%%%%%%%%%%%%%%%%%%%%%%%%%%%%%%%%%%%%
\section{ODEs with a Parameter: Uniform Estimates}\label{app:b:1}
%\subsection{Properties of a certain Class of Functions (Version 1)}

In this section we discuss the behavior of classical solutions $y\in C^2[0,L]$ to the equation
\begin{equation}\label{ode:hom:1}
y''+\frac {\tau^2}{P(x)}y=0, \qquad x\in(0,L),
\end{equation}
where $\tau\in\R$ and $P\in C^1[0,L]$ is a real-valued function satisfying $P^0\le P(x)\le P^1$ uniformly for
$x\in[0,L]$ for some positive constants $P^0, P^1$. Since $\tau$ only occurs squared, we can assume that $\tau\ge 0$
holds in the following.

\begin{lem}[\citealp{ref:ode1}]\label{lem:J}
Let $(\phi_1, \phi_2)$ be an arbitrary pair of linearly independent solutions of (\ref{ode:hom:1}). Then the Green's
function of the equation is given by
\begin{equation}\label{def:J}
J(x,t):=\frac{\phi_1(x)\phi_2(t)-\phi_2(x)\phi_1(t)}{ \phi'_1(t)\phi_2(t)-\phi_1(t)\phi'_2(t)}.
\end{equation}
Furthermore, the Wronskian $W(t):= \phi'_1(t)\phi_2(t)-\phi_1(t)\phi'_2(t)$ is constant for $t\in[0,L]$.
Hence, \eqref{def:J} simplifies to $J(x,t)=C[\phi_1(x)\phi_2(t)-\phi_2(x)\phi_1(t)]$.
\end{lem}

With the prescribed initial data $\varphi(0)$ and $\varphi'(0)$, we shall denote the unique classical solution of
(\ref{ode:hom:1}) by $\varphi_\tau$. The behavior of solutions of (\ref{ode:hom:1}) is stated in the following lemma.
For the proof, see Prop.~2.1 in \citep{abn}.

\begin{lem}\label{sol:est}
There exists a constant $C>0$ such that for any family of solutions $\{\phi_\tau\}_{\tau>1}$ of (\ref{ode:hom:1}) the
following estimates hold uniformly for  $\tau>1$:
\[\|\phi_\tau\|_{L^\infty} \le \frac C\tau\left(\tau|\phi_\tau(0)|+| \phi'_\tau(0)|\right),\qquad
\| \phi'_\tau\|_{L^\infty} \le  C\left(\tau|\phi_\tau(0)|+| \phi'_\tau(0)|\right).\]
\end{lem}

Now we are able to prove the following theorem:

\begin{trm}\label{f_integrl}
Let $\{J_\tau\}_{\tau>1}$ be the family of Green's functions defined in Lemma \ref{lem:J}. Then there exists a constant
$C>0$  such that the following estimates hold uniformly for all $f\in H^1(0,L)$ and $\tau>1$:
\begin{align}
	\left\|\int_0^x f(t)J_\tau(x,t)\d t\right\|_{L^\infty} &\le \frac C{\tau^2}\|f\|_{H^1},\label{int:est:1}\\
	\left\|\int_0^x f(t)\pd_xJ_\tau(x,t)\d t\right\|_{L^\infty} &\le \frac C{\tau}\|f\|_{H^1}.\label{int:est:2}
\end{align}
\end{trm}

\begin{proof}
We are going to show (\ref{int:est:1}), the proof of (\ref{int:est:2}) can be done analogously. The index $\tau$ is
omitted for sake of simplicity. First, we make the substitution $t=x-\xi$ in the left hand integral, and define the
family of functions $\psi_x:\xi\mapsto J(x,x-\xi)$ with parameter $x$. These functions are solutions of the equation
\begin{equation}\label{eq-psi}
\psi''_x+\frac{\tau^2}{P(x-\xi)}\psi_x=0,
\end{equation}
with $'$ denoting here derivatives with respect to $\xi$.
$\psi_x$ takes the initial values $\psi_x(\xi=0)=0$ and $\psi'_x(\xi=0)=1$. Now, integrating by parts yields
\begin{align*}
\left|\int_0^x f(x-\xi)\psi_x(\xi)\d \xi\right| &=
\bigg|-\int_0^x \pd_\xi[(fP)(x-\xi)]\int_0^\xi\!\!\frac{\psi_x(\zeta)}{P(x-\zeta)}
\d\zeta\d\xi\\
&\quad\quad+f(0)P(0)\int_0^x \frac{\psi_x(\zeta)}{P(x-\zeta)}\d\zeta\bigg|\\
	&\le 2\frac{\|\psi'_x\|_{L^\infty}}{\tau^2}\bigg(\int_0^x\!\!|\pd_\xi(fP)(x-\xi)|\d\xi+|f(0)P(0)|\bigg)\\
	&\le C\frac{\|\psi'_x\|_{L^\infty}\|f\|_{H^1}}{\tau^2},
\end{align*}
where we used \eqref{eq-psi} in the second step. And in the last step we used
the continuous embedding $H^1(0,L)\hookrightarrow C[0,L]$. From Lemma \ref{sol:est} and the known initial conditions of
$\psi_x$ we find that $\|\psi'_x\|_{L^\infty}$ is uniformly bounded for all $\tau>1$. Finally, we notice that $\psi_x(x-t)=J(x,t)$, so the above estimate also holds for $J$ instead of $\psi_x$, which proves (\ref{int:est:1}).
\end{proof}

\section{Deferred proofs}\label{app:c}

\begin{proof}[Proof of Lemma~\ref{lem:diss}]
For all $z\in D(A)$ we have:
\begin{align}
	\Re\la z,Az\rah = & \,\,\Re\bigg[{\alpha_1\gamma}\int_0^L(Pw')'(P \bar v')'\d x + \alpha_1\int_0^LP w' \bar v'\d
x\nonumber\\
	&+\alpha_1\gamma P(L)w'(L) \bar v'(L)+\alpha_2w(0)\bar v(0)\nonumber\\
	&+{\alpha_1\gamma}\int_0^LP v'(P \bar w')''\d x+\alpha_1\int_0^Lv(P\bar w')'\d x\nonumber\\
	&-\alpha_1P(L)v(L) \bar w'(L)+\alpha_2\gamma v(0)\bar F\nonumber\\
	&+\frac 12\big[v(0)-2\alpha_1P(0) w'(0)+2\alpha_2w(0)\big]\big[\bar F-2\alpha_1P(0) \bar v'(0)+2\alpha_2\bar
v(0)\big]\bigg]\nonumber\displaybreak[0]\\
	=&\,\,\Re\bigg[{\alpha_1\gamma}\int_0^L\big[Pv'(P \bar w')'\big]'\d x+\alpha_1\int_0^L \big[Pv \bar w'\big]'\d
x\label{diss:1}\\
	&+\alpha_1\gamma P(L)w'(L) \bar v'(L)+\alpha_2w(0)\bar v(0)\nonumber\\
	&-\alpha_1P(L)v(L) \bar w'(L)+\alpha_2\gamma v(0)\bar F\nonumber\\
	&+\frac12\big[v(0)-2\alpha_1P(0) w'(0)+2\alpha_2w(0)\big]\big[\bar F-2\alpha_1P(0) \bar v'(0)+2\alpha_2\bar
v(0)\big]\bigg]\nonumber.
\end{align}
Using the boundary conditions in $D(A)$ to evaluate the term $P v'(P \bar w')'|_0^L$, we find that the real parts of all
terms at $x=L$ cancel against the real part of the third term of (\ref{diss:1}). The remaining terms are
\begin{align*}
	\Re\la z,Az\rah &=\Re\Big[\bar v(0)[-\alpha_1P(0) w'(0)+\alpha_2w(0)]+\gamma \bar F[-\alpha_1P(0)
v'(0)+\alpha_2v(0)]+\\
	&+\frac 12\big[v(0)-2\alpha_1P(0) w'(0)+2\alpha_2w(0)\big]\big[\bar F-2\alpha_1P(0) \bar v'(0)+2\alpha_2\bar
v(0)\big]\Big].
\end{align*}
By introducing the functional $J:w\mapsto -2\alpha_1P(0) w'(0)+2\alpha_2w(0)$, we simplify the expression:
\begin{equation}\label{dis:1}
\Re\la z,Az\rah=\frac 12\Re\Big[\bar v(0)J(w)+\gamma \bar FJ(v)+\big[v(0)+J(w)\big]\big[\bar F+J(\bar v)\big]\Big].\end{equation}
Assuming the relations (\ref{cond:theta}) and (\ref{alphas}) we can write $F=-av(0)-bJ(w)-J(v)$ with $a,b>0$. Then the right hand side of (\ref{dis:1}) only depends on the three independent values $v(0),J(w)$ and $J(v)$.
Introducing the new variables $y_1=\sqrt{a}\,v(0),\, y_2=\sqrt{b}\,J(w)$ and $y_3=\sqrt{\gamma}\,J(v)$ yields:
\begin{equation}\label{dis:2}\Re\la z,Az\rah=\frac 12\big(P_3(\Re y_1,\Re y_2,\Re y_3)+P_3(\Im y_1,\Im y_2,\Im y_3)\big),\end{equation}
where $P_3$ is the polynomial defined by
\begin{equation}P_3(x_1,x_2,x_3):=-x_1^2-x_2^2-x_3^2-2x_1x_2\Big(\frac{a+b-1}{2\sqrt{ab}}\Big)-2x_2x_3\Big(\frac{\sqrt{b\gamma}}{2}
\Big)-2x_1x_3\Big(\frac{\sqrt{a\gamma}}{2}\Big).\end{equation}
Hence, $A$ is dissipative if
\begin{equation}\label{polynom}x_1^2+x_2^2+x_3^2+2x_1x_2\Big(\frac{a+b-1}{2\sqrt{ab}}\Big)+2x_2x_3\Big(\frac{\sqrt{b\gamma}}{2}
\Big)+2x_1x_3\Big(\frac{\sqrt{a\gamma}}{2}\Big)\ge 0,\quad \forall x_1,x_2,x_3\in\R.\end{equation}
According to Lemma \ref{lem:A:2} this inequality is satisfied if there holds:
\[a,b\le\frac 4\gamma,\qquad \frac{(a+b-1)^2}{4ab}\le1,\qquad\frac{(a+b-1)^2}{4ab}\le 1-\frac \gamma4.\]
Since $\gamma>0$ has not yet been specified, we can choose $\gamma$ arbitrarily small, so that the above conditions reduce to the single condition
\begin{equation}\label{cond:1}
\frac{(a+b-1)^2}{4ab}<1.
\end{equation}
So, the relation (\ref{cond:theta}) on the $\theta_i$ together with the condition (\ref{cond:1}) on the $a,b>0$ is sufficient for the dissipativity of $A$ in $\H$ with respect to the inner product (\ref{inner_prod}), with the choice (\ref{alphas}) for $\alpha_1$ and $\alpha_2$, and $\gamma>0$ sufficiently small.
\end{proof}

\begin{center}

\end{center}

\bigskip
\centerline{\bf Acknowledgments}
\

The first author was supported by the FWF-project I395-N16 and the FWF doctoral school ``Dissipation
and dispersion in nonlinear partial differential equations''. The second and the third authors were partially supported by the Doctoral School ``Partial differential equations in technical systems: modeling, simulation, and control'' of Technische Universit\"at Wien. The second author acknowledges a sponsorship by \emph{Clear Sky Ventures}.  We are grateful to the anonymous referee who drew our attention to the references \citep{Conrad:Mifdal,Mifdal:1997}.

\bibliographystyle{apacite} % Verwendung von apacite.sty
\bibliography{references}

\begin{thebibliography}{}

\bibitem [\protect \citeauthoryear {%
Adams%
}{%
Adams%
}{%
{\protect \APACyear {1975}}%
}]{%
ad}
\APACinsertmetastar {%
ad}%
\begin{APACrefauthors}%
Adams, R\BPBI A.%
\end{APACrefauthors}%
\unskip\
\newblock
\APACrefYear{1975}.
\newblock
\APACrefbtitle {Sobolev spaces} {Sobolev spaces}\ (\BVOL~65).
\newblock
\APACaddressPublisher{}{Academic Press, New York, London}.
\PrintBackRefs{\CurrentBib}

\bibitem [\protect \citeauthoryear {%
Arnold%
, {Ben Abdallah}%
\BCBL {}\ \BBA {} Negulescu%
}{%
Arnold%
\ \protect \BOthers {.}}{%
{\protect \APACyear {2011}}%
}]{%
abn}
\APACinsertmetastar {%
abn}%
\begin{APACrefauthors}%
Arnold, A.%
, {Ben Abdallah}, N.%
\BCBL {}\ \BBA {} Negulescu, C.%
\end{APACrefauthors}%
\unskip\
\newblock
\APACrefYearMonthDay{2011}{}{}.
\newblock
{\BBOQ}\APACrefatitle {W{KB}-based schemes for the oscillatory 1{D}
  {S}chr{\"o}dinger equation in the semiclassical limit} {W{KB}-based schemes
  for the oscillatory 1{D} {S}chr{\"o}dinger equation in the semiclassical
  limit}.{\BBCQ}
\newblock
\APACjournalVolNumPages{SIAM J. Numer. Anal.}{49}{4}{1436--1460}.
\PrintBackRefs{\CurrentBib}

\bibitem [\protect \citeauthoryear {%
Birkhoff%
\ \BBA {} Rota%
}{%
Birkhoff%
\ \BBA {} Rota%
}{%
{\protect \APACyear {1962}}%
}]{%
ref:ode1}
\APACinsertmetastar {%
ref:ode1}%
\begin{APACrefauthors}%
Birkhoff, G.%
\BCBT {}\ \BBA {} Rota, G\BHBI C.%
\end{APACrefauthors}%
\unskip\
\newblock
\APACrefYear{1962}.
\newblock
\APACrefbtitle {Ordinary differential equations} {Ordinary differential
  equations}.
\newblock
\APACaddressPublisher{}{Ginn and Company, Boston, New York, Toronto}.
\PrintBackRefs{\CurrentBib}

\bibitem [\protect \citeauthoryear {%
Chentouf%
\ \BBA {} Couchouron%
}{%
Chentouf%
\ \BBA {} Couchouron%
}{%
{\protect \APACyear {1999}}%
}]{%
Chentouf:Couchouron}
\APACinsertmetastar {%
Chentouf:Couchouron}%
\begin{APACrefauthors}%
Chentouf, B.%
\BCBT {}\ \BBA {} Couchouron, J\BHBI F.%
\end{APACrefauthors}%
\unskip\
\newblock
\APACrefYearMonthDay{1999}{}{}.
\newblock
{\BBOQ}\APACrefatitle {Nonlinear feedback stabilization of a rotating body-beam
  without damping} {Nonlinear feedback stabilization of a rotating body-beam
  without damping}.{\BBCQ}
\newblock
\APACjournalVolNumPages{ESAIM: Control, Optimisation and Calculus of
  Variations}{4}{}{515-535}.
\PrintBackRefs{\CurrentBib}

\bibitem [\protect \citeauthoryear {%
Conrad%
\ \BBA {} Mifdal%
}{%
Conrad%
\ \BBA {} Mifdal%
}{%
{\protect \APACyear {1998}}%
}]{%
Conrad:Mifdal}
\APACinsertmetastar {%
Conrad:Mifdal}%
\begin{APACrefauthors}%
Conrad, F.%
\BCBT {}\ \BBA {} Mifdal, A.%
\end{APACrefauthors}%
\unskip\
\newblock
\APACrefYearMonthDay{1998}{}{}.
\newblock
{\BBOQ}\APACrefatitle {Strong stability of a model of an overhead crane}
  {Strong stability of a model of an overhead crane}.{\BBCQ}
\newblock
\APACjournalVolNumPages{Control and Cybernetics}{37}{3}{363--374}.
\PrintBackRefs{\CurrentBib}

\bibitem [\protect \citeauthoryear {%
Conrad%
\ \BBA {} Morg{\"u}l%
}{%
Conrad%
\ \BBA {} Morg{\"u}l%
}{%
{\protect \APACyear {1998}}%
}]{%
conrad1998stabilization}
\APACinsertmetastar {%
conrad1998stabilization}%
\begin{APACrefauthors}%
Conrad, F.%
\BCBT {}\ \BBA {} Morg{\"u}l, {\"O}.%
\end{APACrefauthors}%
\unskip\
\newblock
\APACrefYearMonthDay{1998}{}{}.
\newblock
{\BBOQ}\APACrefatitle {On the stabilization of a flexible beam with a tip mass}
  {On the stabilization of a flexible beam with a tip mass}.{\BBCQ}
\newblock
\APACjournalVolNumPages{SIAM Journal on Control and
  Optimization}{36}{6}{1962--1986}.
\PrintBackRefs{\CurrentBib}

\bibitem [\protect \citeauthoryear {%
Coron%
\ \BBA {} d'Andrea Novel%
}{%
Coron%
\ \BBA {} d'Andrea Novel%
}{%
{\protect \APACyear {1998}}%
}]{%
Coron:Novel}
\APACinsertmetastar {%
Coron:Novel}%
\begin{APACrefauthors}%
Coron, J\BHBI M.%
\BCBT {}\ \BBA {} d'Andrea Novel, B.%
\end{APACrefauthors}%
\unskip\
\newblock
\APACrefYearMonthDay{1998}{}{}.
\newblock
{\BBOQ}\APACrefatitle {Stabilization of a rotating body beam without damping}
  {Stabilization of a rotating body beam without damping}.{\BBCQ}
\newblock
\APACjournalVolNumPages{IEEE Transactions on Automatic
  Control}{43}{5}{608-618}.
\PrintBackRefs{\CurrentBib}

\bibitem [\protect \citeauthoryear {%
d'Andr{\'e}a Novel%
, Boustany%
\BCBL {}\ \BBA {} Conrad%
}{%
d'Andr{\'e}a Novel%
\ \protect \BOthers {.}}{%
{\protect \APACyear {1992}}%
}]{%
MR1173433}
\APACinsertmetastar {%
MR1173433}%
\begin{APACrefauthors}%
d'Andr{\'e}a Novel, B.%
, Boustany, F.%
\BCBL {}\ \BBA {} Conrad, F.%
\end{APACrefauthors}%
\unskip\
\newblock
\APACrefYearMonthDay{1992}{}{}.
\newblock
{\BBOQ}\APACrefatitle {Control of an overhead crane: stabilization of
  flexibilities} {Control of an overhead crane: stabilization of
  flexibilities}.{\BBCQ}
\newblock
\BIn{} \APACrefbtitle {Boundary control and boundary variation} {Boundary
  control and boundary variation}\ (\BVOL~178, \BPGS\ 1--26).
\newblock
\APACaddressPublisher{}{Springer, Berlin}.
\PrintBackRefs{\CurrentBib}

\bibitem [\protect \citeauthoryear {%
d'Andr{\'e}a Novel%
\ \BBA {} Coron%
}{%
d'Andr{\'e}a Novel%
\ \BBA {} Coron%
}{%
{\protect \APACyear {2000}}%
}]{%
MR1828905}
\APACinsertmetastar {%
MR1828905}%
\begin{APACrefauthors}%
d'Andr{\'e}a Novel, B.%
\BCBT {}\ \BBA {} Coron, J\BPBI M.%
\end{APACrefauthors}%
\unskip\
\newblock
\APACrefYearMonthDay{2000}{}{}.
\newblock
{\BBOQ}\APACrefatitle {Exponential stabilization of an overhead crane with
  flexible cable via a back-stepping approach} {Exponential stabilization of an
  overhead crane with flexible cable via a back-stepping approach}.{\BBCQ}
\newblock
\APACjournalVolNumPages{Automatica}{36}{4}{587--593}.
\PrintBackRefs{\CurrentBib}

\bibitem [\protect \citeauthoryear {%
Grabowski%
}{%
Grabowski%
}{%
{\protect \APACyear {2008}}%
}]{%
grab2}
\APACinsertmetastar {%
grab2}%
\begin{APACrefauthors}%
Grabowski, P.%
\end{APACrefauthors}%
\unskip\
\newblock
\APACrefYearMonthDay{2008}{}{}.
\newblock
{\BBOQ}\APACrefatitle {The motion planning problem and exponential
  stabilization of a heavy chain. {II}} {The motion planning problem and
  exponential stabilization of a heavy chain. {II}}.{\BBCQ}
\newblock
\APACjournalVolNumPages{Opuscula Math.}{28}{4}{481--505}.
\PrintBackRefs{\CurrentBib}

\bibitem [\protect \citeauthoryear {%
Grabowski%
}{%
Grabowski%
}{%
{\protect \APACyear {2009}}%
}]{%
grab1}
\APACinsertmetastar {%
grab1}%
\begin{APACrefauthors}%
Grabowski, P.%
\end{APACrefauthors}%
\unskip\
\newblock
\APACrefYearMonthDay{2009}{}{}.
\newblock
{\BBOQ}\APACrefatitle {The motion planning problem and exponential
  stabilisation of a heavy chain. {I}} {The motion planning problem and
  exponential stabilisation of a heavy chain. {I}}.{\BBCQ}
\newblock
\APACjournalVolNumPages{Internat. J. Control}{82}{8}{1539--1563}.
\PrintBackRefs{\CurrentBib}

\bibitem [\protect \citeauthoryear {%
Huang%
}{%
Huang%
}{%
{\protect \APACyear {1985}}%
}]{%
ref:huang}
\APACinsertmetastar {%
ref:huang}%
\begin{APACrefauthors}%
Huang, F\BPBI L.%
\end{APACrefauthors}%
\unskip\
\newblock
\APACrefYearMonthDay{1985}{}{}.
\newblock
{\BBOQ}\APACrefatitle {Characteristic conditions for exponential stability of
  linear dynamical systems in {H}ilbert spaces} {Characteristic conditions for
  exponential stability of linear dynamical systems in {H}ilbert
  spaces}.{\BBCQ}
\newblock
\APACjournalVolNumPages{Ann. Differential Equations}{1}{1}{43--56}.
\PrintBackRefs{\CurrentBib}

\bibitem [\protect \citeauthoryear {%
Kato%
}{%
Kato%
}{%
{\protect \APACyear {1966}}%
}]{%
kato}
\APACinsertmetastar {%
kato}%
\begin{APACrefauthors}%
Kato, T.%
\end{APACrefauthors}%
\unskip\
\newblock
\APACrefYear{1966}.
\newblock
\APACrefbtitle {Perturbation theory for linear operators} {Perturbation theory
  for linear operators}\ (\BVOL~132).
\newblock
\APACaddressPublisher{}{Springer, New York}.
\PrintBackRefs{\CurrentBib}

\bibitem [\protect \citeauthoryear {%
Krsti\'{c}%
, Kanellakopoulos%
\BCBL {}\ \BBA {} Kokotovi\'{c}%
}{%
Krsti\'{c}%
\ \protect \BOthers {.}}{%
{\protect \APACyear {1995}}%
}]{%
refKrstic}
\APACinsertmetastar {%
refKrstic}%
\begin{APACrefauthors}%
Krsti\'{c}, M.%
, Kanellakopoulos, I.%
\BCBL {}\ \BBA {} Kokotovi\'{c}, P.%
\end{APACrefauthors}%
\unskip\
\newblock
\APACrefYear{1995}.
\newblock
\APACrefbtitle {Nonlinear and Adaptive Control Design} {Nonlinear and adaptive
  control design}.
\newblock
\APACaddressPublisher{New York}{John Wiley \& Sons}.
\PrintBackRefs{\CurrentBib}

\bibitem [\protect \citeauthoryear {%
Kugi%
\ \BBA {} Thull%
}{%
Kugi%
\ \BBA {} Thull%
}{%
{\protect \APACyear {2005}}%
}]{%
KT05}
\APACinsertmetastar {%
KT05}%
\begin{APACrefauthors}%
Kugi, A.%
\BCBT {}\ \BBA {} Thull, D.%
\end{APACrefauthors}%
\unskip\
\newblock
\APACrefYearMonthDay{2005}{}{}.
\newblock
{\BBOQ}\APACrefatitle {Infinite-Dimensional Decoupling Control of the Tip
  Position and the Tip Angle of a Composite Piezoelectric Beam with Tip Mass}
  {Infinite-dimensional decoupling control of the tip position and the tip
  angle of a composite piezoelectric beam with tip mass}.{\BBCQ}
\newblock
\BIn{} \APACrefbtitle {Control and Observer Design for Nonlinear Finite and
  Infinite Dimensional Systems} {Control and observer design for nonlinear
  finite and infinite dimensional systems}\ (\BVOL~322, \BPG~351-368).
\newblock
\APACaddressPublisher{Berlin Heidelberg}{Springer}.
\PrintBackRefs{\CurrentBib}

\bibitem [\protect \citeauthoryear {%
Luo%
, Guo%
\BCBL {}\ \BBA {} Morg{\"u}l%
}{%
Luo%
\ \protect \BOthers {.}}{%
{\protect \APACyear {1999}}%
}]{%
lgm}
\APACinsertmetastar {%
lgm}%
\begin{APACrefauthors}%
Luo, Z\BHBI H.%
, Guo, B\BHBI Z.%
\BCBL {}\ \BBA {} Morg{\"u}l, {\"O}.%
\end{APACrefauthors}%
\unskip\
\newblock
\APACrefYear{1999}.
\newblock
\APACrefbtitle {Stability and stabilization of infinite dimensional systems
  with applications} {Stability and stabilization of infinite dimensional
  systems with applications}.
\newblock
\APACaddressPublisher{}{Springer, London}.
\PrintBackRefs{\CurrentBib}

\bibitem [\protect \citeauthoryear {%
Mifdal%
}{%
Mifdal%
}{%
{\protect \APACyear {1997}}%
{\protect \APACexlab {{\protect \BCnt {1}}}}}]{%
Mifdal-diss:1997}
\APACinsertmetastar {%
Mifdal-diss:1997}%
\begin{APACrefauthors}%
Mifdal, A.%
\end{APACrefauthors}%
\unskip\
\newblock
\APACrefYear{1997{\protect \BCnt {1}}}.
\newblock
\APACrefbtitle {Etude de la stabilisation forte et uniforme de syst\`{e}me
  hybride -- Application \`{a} un mod\`{e}le de pont roulant} {Etude de la
  stabilisation forte et uniforme de syst\`{e}me hybride -- application \`{a}
  un mod\`{e}le de pont roulant}.
\newblock
\APACaddressPublisher{}{Dissertation at University Nancy I}.
\PrintBackRefs{\CurrentBib}

\bibitem [\protect \citeauthoryear {%
Mifdal%
}{%
Mifdal%
}{%
{\protect \APACyear {1997}}%
{\protect \APACexlab {{\protect \BCnt {2}}}}}]{%
Mifdal:1997}
\APACinsertmetastar {%
Mifdal:1997}%
\begin{APACrefauthors}%
Mifdal, A.%
\end{APACrefauthors}%
\unskip\
\newblock
\APACrefYearMonthDay{1997{\protect \BCnt {2}}}{}{}.
\newblock
{\BBOQ}\APACrefatitle {Stabilisation uniforme d'un syst\`{e}me hybride}
  {Stabilisation uniforme d'un syst\`{e}me hybride}.{\BBCQ}
\newblock
\APACjournalVolNumPages{C. R. Acad. Sci. Paris}{324}{}{37--42}.
\PrintBackRefs{\CurrentBib}

\bibitem [\protect \citeauthoryear {%
Mileti\'c%
, St{\"u}rzer%
\BCBL {}\ \BBA {} Arnold%
}{%
Mileti\'c%
\ \protect \BOthers {.}}{%
{\protect \APACyear {2015}}%
}]{%
msa_15_1}
\APACinsertmetastar {%
msa_15_1}%
\begin{APACrefauthors}%
Mileti\'c, M.%
, St{\"u}rzer, D.%
\BCBL {}\ \BBA {} Arnold, A.%
\end{APACrefauthors}%
\unskip\
\newblock
\APACrefYearMonthDay{2015}{}{}.
\newblock
{\BBOQ}\APACrefatitle {An {E}uler-{B}ernoulli beam with nonlinear damping and a
  nonlinear spring at the tip} {An {E}uler-{B}ernoulli beam with nonlinear
  damping and a nonlinear spring at the tip}.{\BBCQ}
\newblock
\APACjournalVolNumPages{Discrete Contin. Dyn. Syst. Ser. B}{20}{9}{}.
\PrintBackRefs{\CurrentBib}

\bibitem [\protect \citeauthoryear {%
Mileti\'c%
, St{\"u}rzer%
, Arnold%
\BCBL {}\ \BBA {} Kugi%
}{%
Mileti\'c%
\ \protect \BOthers {.}}{%
{\protect \APACyear {2016}}%
}]{%
MSAK}
\APACinsertmetastar {%
MSAK}%
\begin{APACrefauthors}%
Mileti\'c, M.%
, St{\"u}rzer, D.%
, Arnold, A.%
\BCBL {}\ \BBA {} Kugi, A.%
\end{APACrefauthors}%
\unskip\
\newblock
\APACrefYearMonthDay{2016}{}{}.
\newblock
{\BBOQ}\APACrefatitle {{Stability of an {E}uler-{B}ernoulli beam with a
  nonlinear dynamic feedback system}} {{Stability of an {E}uler-{B}ernoulli
  beam with a nonlinear dynamic feedback system}}.{\BBCQ}
\newblock
\APACjournalVolNumPages{IEEE Transactions on Automatic
  Control}{61}{10}{2782--2795}.
\PrintBackRefs{\CurrentBib}

\bibitem [\protect \citeauthoryear {%
Morg{\"u}l%
}{%
Morg{\"u}l%
}{%
{\protect \APACyear {2001}}%
}]{%
Morgul2001}
\APACinsertmetastar {%
Morgul2001}%
\begin{APACrefauthors}%
Morg{\"u}l, {\"O}.%
\end{APACrefauthors}%
\unskip\
\newblock
\APACrefYearMonthDay{2001}{}{}.
\newblock
{\BBOQ}\APACrefatitle {{Stabilization and Disturbance Rejection for the Beam
  Equation}} {{Stabilization and Disturbance Rejection for the Beam
  Equation}}.{\BBCQ}
\newblock
\APACjournalVolNumPages{IEEE Transactions on Automatic
  Control}{46}{12}{1913--1918}.
\PrintBackRefs{\CurrentBib}

\bibitem [\protect \citeauthoryear {%
Morg{\"u}l%
, Rao%
\BCBL {}\ \BBA {} Conrad%
}{%
Morg{\"u}l%
\ \protect \BOthers {.}}{%
{\protect \APACyear {1994}}%
}]{%
Morgul1994}
\APACinsertmetastar {%
Morgul1994}%
\begin{APACrefauthors}%
Morg{\"u}l, {\"O}.%
, Rao, B\BPBI P.%
\BCBL {}\ \BBA {} Conrad, F.%
\end{APACrefauthors}%
\unskip\
\newblock
\APACrefYearMonthDay{1994}{}{}.
\newblock
{\BBOQ}\APACrefatitle {{On the Stabilization of a Cable with a Tip Mass}} {{On
  the Stabilization of a Cable with a Tip Mass}}.{\BBCQ}
\newblock
\APACjournalVolNumPages{IEEE Transactions on Automatic
  Control}{39}{10}{2140--2145}.
\PrintBackRefs{\CurrentBib}

\bibitem [\protect \citeauthoryear {%
Pazy%
}{%
Pazy%
}{%
{\protect \APACyear {1983}}%
}]{%
pazy}
\APACinsertmetastar {%
pazy}%
\begin{APACrefauthors}%
Pazy, A.%
\end{APACrefauthors}%
\unskip\
\newblock
\APACrefYear{1983}.
\newblock
\APACrefbtitle {Semigroups of linear operators and applications to partial
  differential equations} {Semigroups of linear operators and applications to
  partial differential equations}\ (\BVOL~44).
\newblock
\APACaddressPublisher{}{Springer, New York}.
\PrintBackRefs{\CurrentBib}

\bibitem [\protect \citeauthoryear {%
Petit%
\ \BBA {} Rouchon%
}{%
Petit%
\ \BBA {} Rouchon%
}{%
{\protect \APACyear {2001}}%
}]{%
Petit:Rouchon}
\APACinsertmetastar {%
Petit:Rouchon}%
\begin{APACrefauthors}%
Petit, N.%
\BCBT {}\ \BBA {} Rouchon, P.%
\end{APACrefauthors}%
\unskip\
\newblock
\APACrefYearMonthDay{2001}{}{}.
\newblock
{\BBOQ}\APACrefatitle {Flatness of heavy chain systems} {Flatness of heavy
  chain systems}.{\BBCQ}
\newblock
\APACjournalVolNumPages{SIAM Journal on Control and
  Optimization}{40}{2}{475--495}.
\PrintBackRefs{\CurrentBib}

\bibitem [\protect \citeauthoryear {%
Thull%
, Wild%
\BCBL {}\ \BBA {} Kugi%
}{%
Thull%
\ \protect \BOthers {.}}{%
{\protect \APACyear {2005}}%
}]{%
Thull:Wild:Kugi:at}
\APACinsertmetastar {%
Thull:Wild:Kugi:at}%
\begin{APACrefauthors}%
Thull, D.%
, Wild, D.%
\BCBL {}\ \BBA {} Kugi, A.%
\end{APACrefauthors}%
\unskip\
\newblock
\APACrefYearMonthDay{2005}{}{}.
\newblock
{\BBOQ}\APACrefatitle {Infinit-dimensionale {R}egelung eines
  {B}r{\"u}ckenkranes mit schweren {K}etten} {Infinit-dimensionale {R}egelung
  eines {B}r{\"u}ckenkranes mit schweren {K}etten}.{\BBCQ}
\newblock
\APACjournalVolNumPages{at-Automatisierungstechnik}{53}{8}{400--410}.
\PrintBackRefs{\CurrentBib}

\bibitem [\protect \citeauthoryear {%
Thull%
, Wild%
\BCBL {}\ \BBA {} Kugi%
}{%
Thull%
\ \protect \BOthers {.}}{%
{\protect \APACyear {2006}}%
}]{%
TWK06}
\APACinsertmetastar {%
TWK06}%
\begin{APACrefauthors}%
Thull, D.%
, Wild, D.%
\BCBL {}\ \BBA {} Kugi, A.%
\end{APACrefauthors}%
\unskip\
\newblock
\APACrefYearMonthDay{2006}{}{}.
\newblock
{\BBOQ}\APACrefatitle {Application of a combined flatness- and passivity-based
  control concept to a crane with heavy chains and payload} {Application of a
  combined flatness- and passivity-based control concept to a crane with heavy
  chains and payload}.{\BBCQ}
\newblock
\BIn{} \APACrefbtitle {IEEE Conference on Control Applications (CCA)} {Ieee
  conference on control applications (cca)}\ (\BPG~656–-661).
\PrintBackRefs{\CurrentBib}

\bibitem [\protect \citeauthoryear {%
Yosida%
}{%
Yosida%
}{%
{\protect \APACyear {1980}}%
}]{%
yosida}
\APACinsertmetastar {%
yosida}%
\begin{APACrefauthors}%
Yosida, K.%
\end{APACrefauthors}%
\unskip\
\newblock
\APACrefYear{1980}.
\newblock
\APACrefbtitle {Functional analysis} {Functional analysis}\
  (\PrintOrdinal{Sixth}\ \BEd, \BVOL~123).
\newblock
\APACaddressPublisher{}{Springer, Berlin}.
\PrintBackRefs{\CurrentBib}

\end{thebibliography}

\end{document}